\documentclass[11pt]{elsarticle}

\usepackage[T1]{fontenc}

\usepackage{amsfonts}
\usepackage{amsthm}
\usepackage{amsmath}
\usepackage{amssymb}
\usepackage{bm}
\usepackage{wasysym}
\usepackage{accents}
\usepackage[draft,margin=false,inline=true,author=]{fixme}
\fxusetheme{colorsig}

\usepackage{epsfig}
\usepackage[percent]{overpic}
\usepackage{psfrag}

\usepackage{titlesec}
\titleformat{\paragraph}
{\normalfont\normalsize\bfseries}{\theparagraph}{1em}{}
\titlespacing*{\paragraph}
{0pt}{3.25ex plus 1ex minus .2ex}{1.5ex plus .2ex}

\usepackage{verbatim}
\usepackage{color}
\usepackage{tikz}
\usetikzlibrary{arrows}
\usepackage{tcolorbox}
\usepackage{todonotes}

\usepackage{xcolor,soul}
\usepackage{graphicx}
\usepackage{booktabs}
\usepackage{todonotes}
\usepackage{float}
\usepackage{latexsym}
\usepackage{tabu}
\usepackage{tcolorbox}
\usepackage{multirow}
\usepackage{varwidth}

\usepackage{subcaption}

\definecolor{royalpurple}{rgb}{0.47, 0.32, 0.66}
\definecolor{pastelgreen}{rgb}{0.47, 0.87, 0.47}
\definecolor{cornellred}{rgb}{0.7, 0.11, 0.11}
\definecolor{pastelorange}{rgb}{1.0, 0.7, 0.28}
\definecolor{darkred}{rgb}{0.55, 0.0, 0.0}
\definecolor{darkpastelgreen}{rgb}{0.01, 0.75, 0.24}
\definecolor{blun}{cmyk}{0.8, 0.5, 0, 0.7}
\definecolor{darkcyan}{rgb}{0.0, 0.55, 0.55}

\usepackage{placeins}

\usepackage[all]{xy}
\usepackage{makeidx}
\usepackage{enumerate}
\usepackage[letterpaper,margin=0.9in]{geometry}
\usepackage[bookmarks,colorlinks,pdfauthor={GianninoIuorioSoresina},linkcolor=black,citecolor=blun,urlcolor=black]{hyperref}

\usepackage{booktabs, makecell}

\theoremstyle{plain}
\newtheorem{thm}{Theorem}[section]

\newtheorem{prop}[thm]{Proposition}

\theoremstyle{definition}

\newtheoremstyle{myremark}
  {3pt}
  {3pt}
  {\small \rmfamily}
  {5pt}
  {\rmfamily}
  {:}
  {.5em}
  {}

\theoremstyle{myremark}
\newtheorem*{remark}{\textit{Remark}}

\newcommand{\be}{\begin{equation}}
\newcommand{\ee}{\end{equation}}
\newcommand{\benn}{\begin{equation*}}
\newcommand{\eenn}{\end{equation*}}
\newcommand{\bea}{\begin{eqnarray}}
\newcommand{\eea}{\end{eqnarray}}
\newcommand{\beann}{\begin{eqnarray*}}
\newcommand{\eeann}{\end{eqnarray*}}

\newcommand{\myendex}{$\blacklozenge$\end{ex}}
\newcommand{\myendexerc}{$\lozenge$\end{exerc}}
\newcommand{\myendpexerc}{$\lozenge$\end{pexerc}}

\journal{Physica D}

\title{Beyond water limitation in vegetation--autotoxicity patterning: \\ a cross-diffusion model}

\author[label1]{Francesco Giannino}\ead{giannino@unina.it}
\author[label2]{Annalisa Iuorio}\ead{annalisa.iuorio@uniparthenope.it}
\author[label3]{Cinzia Soresina}\ead{cinzia.soresina@unitn.it}

\affiliation[label1]{organization={University of Naples Federico II, Department of Agricultural Sciences},
            addressline={via Universit{\`a} 100}, 
            city={Portici (NA)},
            postcode={80055},
            country={Italy}}
            
\affiliation[label2]{organization={Parthenope University of Naples, Department of Engineering},
            addressline={Centro Direzionale Isola C4}, 
            city={Naples (NA)},
            postcode={80143},
            country={Italy}}
            
\affiliation[label3]{organization={University of Trento, Department of Mathematics},
            addressline={via Sommarive 14},
            city={Povo (TN)},
            postcode={38123},
            country={Italy}}
            
\begin{document}
\listoffixmes{}

\begin{frontmatter}

\begin{abstract}


\noindent
Many mathematical models describing vegetation patterns are based on biomass--water interactions, due to the impact of this limited resource in arid and semi-arid environments. However, in recent years, a novel biological factor called autotoxicity has proved to play a key role in vegetation spatiotemporal dynamics, particularly by inhibiting biomass growth and increasing its natural mortality rate. In a standard reaction-diffusion framework, biomass-toxicity dynamics alone are unable to support the emergence of stable spatial patterns. \\ 
In this paper, we derive a cross-diffusion model for biomass and toxicity dynamics as the fast-reaction limit of a three-species system involving dichotomy and different time scales. Within this general framework, in addition to growth inhibition and extra-mortality already considered in previous studies, the additional effect of ``propagation reduction'' induced by autotoxicity on vegetation dynamics is obtained.  By combining linearised analysis, simulations, and continuation, we investigate the formation of spatial patterns. Thanks to the cross-diffusion term, for the first time, a spatial model based solely on biomass--toxicity feedback without explicit water dynamics supports the formation of stable (Turing) vegetation patterns for a wide range of parameter values.

\end{abstract}

\begin{keyword}
multiple scales \sep pattern formation \sep autotoxicity \sep cross-diffusion \sep fast-reaction limit \sep bifurcations \sep climate.

\vspace{.2cm}
\MSC[2020] 35K57 \sep 35B36 \sep 35B32 \sep 35Q92 \sep 65P30 \sep 92D40.
\end{keyword}

\end{frontmatter}

\section{Introduction}



Protecting, restoring and promoting sustainable use of terrestrial ecosystems are current challenges of uttermost importance. 
In this framework, vegetation plays a key role in preventing soil degradation and preserving ecosystems' resilience by acting as an ecological indicator: variations in its behaviour -- for instance, pattern formation -- can reveal the proximity of an ecosystem to a catastrophic shift, such as desertification (see e.g.~\cite{Kefi_2007, Kefi_2014}). 

Among others, mathematical models based on partial differential equations (PDEs) have proven to be invaluable tools to improve our understanding of the emergence of vegetation patterns and their link with so-called tipping points (see e.g.~\cite{Bastiaansen_2020_2, Martinez_Garcia_2023, Rietkerk_2021}). In particular, these models support the existence of different kinds of stationary and travelling patterns on both flat and sloped terrains, including front invasions, spots and gaps, and vegetation bands \cite{Jaibi_2020}. Most of these models focus on the interplay between biomass and water, due to the importance of this resource especially in arid environments (see \cite{Bastiaansen_2020_1, Bastiaansen_2019_1, Carter_2018, Eigentler_2020, Gandhi_2021, Gandhi_2018, Gilad_2007, Klausmeier_1999, Meron_2012, Sherratt.2010, von_Hardenberg_2001}, just to name a few). However, an additional ecological element has proved to play a crucial role in shaping plant communities even when water is not scarce, namely autotoxicity (see e.g.~\cite{Abbas_2025, Bonanomi_2014, Carteni_2012, Consolo_2023, Marasco_2020, Marasco_2014, Mazzoleni_2010, Mazzoleni_2014}). Among the hypotheses to explain autotoxicity, we consider here self-DNA due to litter decomposition as it has been widely acclaimed in recent years \cite{Mazzoleni_2015, Mazzoleni_2014}. This factor is, in fact, a key component of plant--soil negative feedback, able to foster biodiversity and justify the presence of several types of patterns observed in environments where water is not limited \cite{Allegrezza_2022, Salvatori_2023}. 
From the theoretical viewpoint, including an equation for autotoxicity is also able to simplify the construction of travelling pulse solutions using Geometric Singular Perturbation Theory \cite{grifo2025far}. Because of its relevant impact on vegetation dynamics, autotoxicity has been added to the standard biomass--water modelling approach, revealing the presence of dynamic, asymmetric patterns not present in the classical Klausmeier framework \cite{Abbas_2024, Carter_2023_2, Iuorio_2021, Marasco_2014}. On the other hand, reaction--diffusion models based only on biomass--water feedback cannot lead to the emergence of stable Turing patterns supported instead by the PDE models based on biomass--water interactions. A natural research question then arises: could a more detailed mathematical (and at the same time ecologically accurate) description of the interplay between biomass and autotoxicity lead to stable Turing patterns even without an explicit water dependence? In other words: can plant--soil negative feedback alone induce the emergence of Turing patterns, if appropriately modelled?

\vspace{.2cm}

This question can be rephrased in the context of pattern formation in reaction--diffusion models as: given a (simple) reaction part that does not present the activator--inhibitor structure for Turing instability, is it possible to modify the diffusion part to obtain stable non-homogenous solutions (Turing patterns)? In this regard, cross-diffusion terms may be the key ingredient, as revealed in several cases including e.g.~the SKT model for competing species~\cite{Fanelli_2013, Ritchie_2022, shigesada1979spatial, villar2024}. In particular, an interesting feature is that cross-diffusion terms can be obtained in different contexts as the singular limit of a ``mesoscopic'' (in terms of level of details in the description of the population interactions and biological processes) model incorporating a dichotomy in a species and multiple time-scales \cite{desvillettes2018, eliavs2021singular, eliavs2022aggregation, iida2006diffusion}. In the context of vegetation dynamics, the presence of autotoxicity induces necrosis in the plant's roots system \cite{Bonanomi_2022, Mazzoleni_2014}, hence reducing its propagation ability. In mathematical terms, by introducing this biological mechanism at the \textit{mesoscopic} scale (namely in the fast-reaction system), we obtain the \textit{macroscopic} model as its singular limit. Thanks to this approach in addition to growth inhibition and extra-mortality already considered in previous studies, the additional effect of ``propagation reduction'' induced by autotoxicity on vegetation dynamics is obtained. Thanks to the cross-diffusion term, for the first time, a spatial model based solely on biomass--water feedback without explicit water dynamics supports the formation of stable (Turing) vegetation patterns for a wide range of parameter values.
\vspace{.2cm}

The paper is organised as follows. In Section~\ref{sec:model}, we present the ``mesoscopic'' reaction--diffusion model and its fast-reaction limit. The emergence of non-homogeneous steady states, investigated by linearised analysis, is presented in Section~\ref{sec:TI}. A numerical investigation of the effect induced by autotoxicity on the stable patterns of biomass by means of propagation-reduction, growth-inhibition, and extra-mortality for a set of biologically meaningful parameters is illustrated in Section~\ref{sec:sim}, together with the bifurcation diagrams obtained by exploiting the continuation software \texttt{pde2path}. Finally, Section~\ref{sec:conclusion} is devoted to discussion and concluding remarks.

\section{The \textit{mesoscopic} model and its fast-reaction limit}\label{sec:model}
The goal of this section is to obtain a macroscopic model exploiting dichotomy and time-scale arguments. To this aim, we fix a sufficiently regular spatial domain $\Omega \subset \mathbb{R}^n, \, n=1,2$. Our state variables consist of the roots biomass and the concentration of autotoxicity (for simplicity named toxicity throughout the paper) in the soil; in particular, we consider two states for the roots, namely \emph{healthy} (non-exposed) and \emph{exposed} to the toxicity. These variables are denoted by $R_h=R_h(x,t)$, $R_e=R_e(x,t)$, and $T=T(x,t)$, respectively, where $x\in\Omega$ and $t>0$. The total roots biomass is indicated by $R=R_h+R_e$.

The growth of both healthy and exposed roots is modelled by a logistic function with reference total biomass $\hat{R}$. The maximum growth rate of the healthy roots is denoted by $g$, while it is assumed that exposed roots have a smaller maximum growth rate denoted by $g-\gamma$ where $0\leq \gamma \leq g$. We also consider a natural root mortality rate $d$ for both healthy and exposed roots and an additional mortality term for exposed roots with a rate $s$ because of exposure to toxicity. Based on biological considerations, we assume $g>d$. In addition to these processes, we also consider a possible switch between healthy and exposed roots: healthy roots exposed to toxicity pass to the exposed state, while roots exposed to toxicity can recover if the toxicity level is low and their exposure time is short enough. In particular, we assume that the state switch happens on a faster time scale than growth, mortality, and propagation. These assumptions are modelled by the time-scale parameter $\varepsilon\ll 1$ and the non-negative toxicity-dependent transition rates $p$ and $q$. In detail, for their biological meaning, we assume that $p$ is an increasing function of $T$, while $q$ is a decreasing function of $T$. Moreover, the diffusion terms in the roots equations account for their propagation into the soil; we assume that the diffusion coefficient of the healthy state is larger than the one for the exposed roots, namely $0\leq \sigma<d_R$.
Regarding the toxicity dynamics, we consider its propagation into the soil with diffusion coefficient $d_T$, a growth term given by the decomposing biomass of healthy and exposed roots with conversion factor $c$, and a decay term with rate $k$. It is assumed that the model parameters are nonnegative quantities.


\begin{figure}[H]
\centering
\begin{tikzpicture}[even odd rule, line width=1pt, x=10pt,y=10pt]
		\draw[rounded corners, fill=gray!30!white] (-7,5) rectangle (21,11);
		\draw[rounded corners, fill=gray!30!white] (4,-3) rectangle (10,1);
		\draw[rounded corners, fill=white] (-6,6) rectangle (0,10);
		\draw[rounded corners, fill=white] (14,6) rectangle (20,10);
        
		\draw[<-, dashed] (7,6.5) -- (7,1);
		\draw[-to] (7,-3) -- (7,-5);
		\draw node at (7,11.8)[] {Roots ($R$)};
		\draw node at (-3,8.5)[] {healthy};
            \draw node at (-3,7.2)[] {($R_h$)};
		\draw node at (17,8.5)[] {exposed};
            \draw node at (17,7.2)[] {($R_e$)};
		\draw node at (7,-0.5)[] {Toxicity};
		\draw node at (7,-2)[] {($T$)};
        
            \draw node at (7,8.8)[] {\scriptsize{exposed/recovered}};
            \draw node at (7,7.2)[] {\scriptsize{roots production}};
            \draw[<->] (0,8) -- (14,8);

            \draw [->] (-3,6) to [out=270,in=180] (4,-1);
            \draw [->] (17,6) to [out=-90,in=0] (10,-1);

            \draw[fill=white,draw=none] (-4,1.3) rectangle (0,3.5);
            \draw[fill=white,draw=none] (14,1.3) rectangle (20,3.5);
            
            \draw node at (16,3)[] {\scriptsize{biomass death}};
            \draw node at (16,2)[] {\scriptsize{toxicity production}};
            \draw node at (-2,3)[] {\scriptsize{biomass death}};
            \draw node at (-2,2)[] {\scriptsize{toxicity production}};

	\end{tikzpicture}
\caption{Schematic representation of the fast-reaction model in Eq.~\eqref{eq:fastreaction} describing the interactions and feedback between healthy ($R_h$) and exposed ($R_e$) roots and toxicity ($T$).} 
\label{fig:model} 
\end{figure}
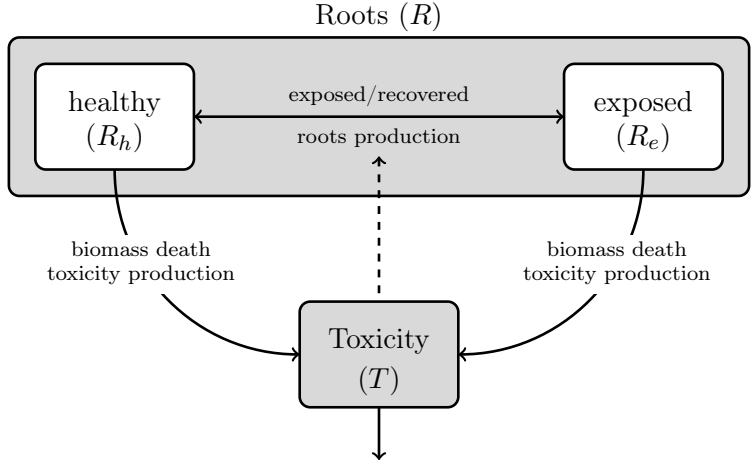

The above interactions between the three state variables, also summarised in Figure~\ref{fig:model}, are described by the following system of PDEs
\begin{equation}
\begin{cases}
\begin{aligned}
\partial_t R_h &=d_R \Delta_x R_h 
    +g R_h \left(1-\dfrac{R}{\hat{R}}\right)-dR_h-\dfrac{1}{\varepsilon}\left(p(T)R_h-q(T)R_e\right), & x \in \Omega, \, t>0,\\[0.3cm]
\partial_t R_e&=(d_R-\sigma) \Delta_x R_e 
    +(g-\gamma) R_e \left(1-\dfrac{R}{\hat{R}}\right) -(d+s)R_e+\dfrac{1}{\varepsilon}\left(p(T)R_h-q(T)R_e\right),& x \in \Omega, \, t>0,\\[0.3cm]
\partial_t T&= d_T \Delta_x T 
    +c(dR_h+dR_e+sR_e)-kT, &x \in \Omega, \, t>0,
\end{aligned}
\end{cases}\label{eq:fastreaction}
\end{equation}
together with homogeneous Neumann boundary conditions and nonnegative initial conditions respectively given by
\begin{equation} \label{eq:BIC}
\begin{aligned}
    &\partial_{\hat{n}} R_h(x,t) = \partial_{\hat{n}} R_e(x,t) = \partial_{\hat{n}} T(x,t) = 0, \qquad &&x \in \partial \Omega, \, t>0, \\
    &R_h(x,0) = R_h^0(x), \quad R_e(x,0) = R_e^0(x), \quad T(x,0) = T^0(x), \qquad &&x \in \Omega,
\end{aligned}
\end{equation}
being $\hat{n}$ the outer normal vector to the boundary of $\Omega$. 
Variables and parameters appearing in the model equations, their biological meaning and units are listed in Table~\ref{tab:pars}. 
\begin{table}[!ht]
    \centering
    \begin{tabular}{c|l|c}
    \toprule
    \thead{\textbf{Quantity}} & \thead{\textbf{Biological interpretation}} & \thead{\textbf{Units}} \\
    \toprule
        $R$ & roots biomass density, $R=R_e+R_h$ & ${\rm kg} \, {\rm m}^{-2}$\\
        $R_h$ & healthy roots biomass density & ${\rm kg} \, {\rm m}^{-2}$\\
        $R_e$ & exposed roots biomass density & ${\rm kg} \, {\rm m}^{-2}$\\
        $\hat{R}$ & reference total roots biomass density & ${\rm kg} \, {\rm m}^{-2}$\\
        $T$ & toxicity concentration & ${\rm g} \, {\rm m}^{-2}$\\
        $\hat{T}$ & critical toxicity concentration & ${\rm g} \, {\rm m}^{-2}$\\
        $g$ & roots biomass maximum growth rate & ${\rm year}^{-1}$ \\
        $\gamma$ & roots biomass growth rate reduction due to toxicity & ${\rm year}^{-1}$ \\
        $d$ & roots biomass mortality rate & ${\rm year}^{-1}$ \\
        $s$ & extra mortality of roots biomass due to toxicity & ${\rm year}^{-1}$ \\
        $c$ & conversion factor of decomposed roots biomass into toxicity & ${\rm g} \, {\rm kg}^{-1}$ \\
        $k$ & toxicity decay rate & ${\rm year}^{-1}$ \\
        $d_R$ & roots propagation coefficient & ${\rm m}^2 \, {\rm year}^{-1}$ \\
        $\sigma$ & reduction of roots propagation coefficient due to toxicity & ${\rm m}^2 \, {\rm year}^{-1}$ \\
        $d_T$ & toxicity diffusion coefficient & ${\rm m}^2 \, {\rm year}^{-1}$ \\
        $\varepsilon$ & time-scale parameter regulating the switch between healthy and exposed roots & ${\rm year}$ \\
        $p(.),\, q(.)$ & toxicity-dependent transition rates & --\\
        \bottomrule
    \end{tabular}
    \caption{Variables and parameters appearing in the model equations~\eqref{eq:fastreaction}, their biological meaning and units.}
    \label{tab:pars}
\end{table}


We now consider the fast-reaction limit of System~\eqref{eq:fastreaction}. At a formal level, when $\varepsilon\to 0$, the fast-reaction part tends to a quasi-steady-state, namely we expect
$$p(T)R_h-q(T)R_e=0.$$
Since $R=R_h+R_e$, we can express $R_h$ and $R_e$ in terms of $R,\, T$ and of the transition rates $p(T),\, q(T)$, obtaining
\begin{equation}
R_h=\dfrac{q(T)}{p(T)+q(T)}R,\qquad R_e=\dfrac{p(T)}{p(T)+q(T)}R.
\end{equation}
Summing the equations for $R_h$ and $R_e$ in Eq.~\eqref{eq:fastreaction}, we obtain an equation for the total roots biomass~$R$, reducing the fast-reaction system to a two-species limiting system. In detail, denoting
\begin{equation}
  \theta(T):=\frac{p(T)}{p(T)+q(T)}, 
\end{equation}\label{eq:theta_pq}
the limiting system reads
\begin{equation} \label{eq:mainsys}
\begin{cases}
\partial_t R&= \Delta_x \left(\left(d_R -\sigma \theta(T)\right)R\right)
+\left(g-\gamma \theta(T)\right)R\left(1-\dfrac{R}{\hat{R}}\right)
-\left(d+s\theta(T)\right)R,\\[0.3cm]
\partial_t T&= d_T \Delta_x T +c\left(d+s\theta(T)\right)R-kT,
\end{cases}
\end{equation}
equipped with the following homogeneous Neumann boundary conditions and initial conditions
\begin{equation} \label{eq:BICfr}
\begin{aligned}
    &\partial_n R(x,t) = \partial_n T(x,t) = 0, \qquad &&x \in \partial \Omega, \, t>0, \\
    &R(x,0) = R^0(x), \quad T(x,0) = T^0(x), \qquad &&x \in \Omega.
\end{aligned}
\end{equation}
Note that the limiting system \eqref{eq:mainsys} presents a cross-diffusion term in the equation for $R$ due to $T$, as well as growth-inhibition and extra-mortality effects due to toxicity. In particular, it is worthwhile to note that the cross-diffusion term incorporates the dichotomy and fast-switching mechanism between healthy and exposed roots combined with reduced propagation of the exposed roots.

\begin{remark}
    A rigorous study of the fast-reaction system~\eqref{eq:fastreaction} and the cross-diffusion system~\eqref{eq:mainsys}, as well as of the convergence of the solution of the fast-reaction to the cross-diffusion system, has been carried out in \cite{morgan2024singular}.
\end{remark} 

\begin{remark}
    Alternatively, one can assume in the fast-reaction system that the propagation coefficient for the healthy roots is higher than the one for the exposed roots; the cross-diffusion limiting system is identical.
\end{remark}

The transition rates $p(T)$ and $q(T)$ determine the function $\theta(T)$ in the limiting cross-diffusion system. Due to their biological meaning, we consider
\begin{equation}
\begin{aligned}
&p(0)=0,\qquad  p(T)\to 1 \quad \textnormal{ when }T\to+\infty, \qquad \dfrac{\textnormal{d}p}{\textnormal{d}T}>0,\\
&q(0)=1, \qquad q(T)\to 0 \quad \textnormal{ when } T\to+\infty, \qquad \dfrac{\textnormal{d}q}{\textnormal{d}T}<0.
\end{aligned}
\label{eq:properties_pq}
\end{equation}
A possible choice fulfilling the above conditions (and adopted in the rest of the paper) is provided by
\begin{equation}
q(T)=1-p(T),\quad 
p(T)=\theta(T)=\begin{cases}
    T/\hat{T},&T\leq \hat{T},\\
    1, & T> \hat{T},
\end{cases}
\label{eq:pqt}
\end{equation}
where $\hat{T}$ denotes a critical concentration of toxicity. 
The precise value of $\hat{T}$ will be fixed in Eq.~\eqref{eq:T_hat} and clarified in Prop.~\ref{prop:pos}. 
Note that, being $\sigma<d_R$, this choice ensures that, in the diffusion term, $d_R-\sigma \theta(T)>0$.
We also finally remark that in the following we assume that the roots dynamics (below ground) correspond to/are a proxy for the vegetation dynamics (above ground).


\section{Cross-diffusion-induced instability}\label{sec:TI}

In this section, we derive the conditions under which cross-diffusion-induced instability occurs. To this aim, we first study the homogeneous model, focusing on the existence and stability of the equilibrium points. Subsequently, we investigate the destabilisation of the homogeneous state in the presence of diffusion; this analysis will be confirmed by numerical simulations and broadened with the support of the continuation software \texttt{pde2path}~\cite{Uecker_2021} in Sec.~\ref{sec:sim}.
\subsection{Homogeneous model}
In the absence of diffusion, system~\eqref{eq:mainsys} reads
\begin{equation} \label{eq:mainsyshom}
\begin{cases}
 \dot{R}&= 
\left(g-\gamma \theta(T)\right)R\left(1-\dfrac{R}{\hat{R}}\right)
-\left(d+s\theta(T)\right)R,\\[0.3cm]
\dot{T}&= c\left(d+s\theta(T)\right)R-kT,   
\end{cases}
\end{equation}
where $\dot{}={\rm d}/{\rm d}t$ and the corresponding initial conditions are given by
\begin{equation} \label{eq:IChom}
    R(0)=R^0, \quad T(0)=T^0,
\end{equation}
being $R^0, \, T^0$ nonnegative quantities.
In particular, we consider the function $\theta$ as in Eq.~\eqref{eq:pqt}. 
For this model, we show the following relevant results.
\begin{prop}[Positivity] \label{prop:pos}
    The solutions $R(t)$, $T(t)$ to Eq.~\eqref{eq:mainsyshom}--\eqref{eq:IChom} where $R^0, T^0\geq 0$ remain nonnegative for all times, i.e.~$R(t)$, $T(t)\geq 0$ $\forall t > 0$. In particular, assuming 
    \begin{equation}
    \hat{T}:=\dfrac{c(d+s)}{k}\hat{R}, \quad with \quad \dfrac{c(d+s)}{k}<1,
\label{eq:T_hat}
\end{equation}
the region $[0, \hat{R}] \times [0, \hat{T}]$ is positively invariant. 
\end{prop}
\begin{proof}
In order to show the result, we start by investigating the dynamics of Eq.~\eqref{eq:mainsyshom}--\eqref{eq:IChom} in phase space on both axes. We observe that:
\begin{itemize}
\item[-] the $T$-axis is a trajectory along which solutions decay to the origin: considering $R=0$ in Eq.~\eqref{eq:mainsyshom}, in fact, leads to $\dot{R}=0$ and $\dot{T}=-kT<0$;
\item[-] on the $R$-axis, the dynamics of $T$ are monotonically increasing: 
for $T=0$, in fact, Eq.~\eqref{eq:mainsyshom} implies then $\dot{T}=cdR>0$. As for $R$, we observe that on the $R$ axis its evolution reads $\dot{R}=gR\left(1-\dfrac{R}{\hat{R}}\right)-dR$. Then, in view of the assumption $g>d$ illustrated in the model description, we introduce the threshold
\begin{equation} \label{eq:Rtilde}
    \tilde{R}:=\frac{(g-d)}{g} \hat{R}
\end{equation}
such that $\dot{R}>0$ if $0<R<\tilde{R}$ and $\dot{R}<0$ if $R > \tilde{R}$.
\end{itemize}
Therefore, trajectories starting with positive initial conditions remain positive for all times.

We now consider the function $\theta$ as in Eq.~\eqref{eq:pqt}. We can thus derive that the set $[0, \hat{R}] \times [0, \hat{T}]$ is invariant from the following observations. First, setting $R=\hat{R}$ in Eq.~\eqref{eq:mainsyshom} leads to $\dot{R}=-\left(d+s\theta(T)\right)\hat{R}<0$. Therefore, if the initial condition satisfies $R^0<\hat{R}$, then $R(t)<\hat{R}$ for all $t>0$.  As for $T$, when $T=\hat{T}$ we have $\dot{T}=c(d+s)R-k \hat{T}$. Consequently, in view of the fact that $R(t) < \hat{R}$ $\forall t > 0$, we obtain that
\[
\dot{T}=c(d+s)R-k \hat{T} < c(d+s)\hat{R}-k \hat{T} \leq 0
\]
as long as the critical threshold $\hat{T}$ satisfies
    $\hat{T}\geq c(d+s)\hat{R}/{k} $.
We can thus fix the value of $\hat{T}$ as the infimum of the toxicity thresholds for which this condition holds, i.e.,
\begin{equation}
    \hat{T}:=\dfrac{c(d+s)}{k}\hat{R}, \quad \textnormal{ assuming } \quad \dfrac{c(d+s)}{k}<1.
\label{eq:T_hat}
\end{equation}
This assumption is posed to ensure the biologically feasible condition that the toxicity density never exceeds the biomass density.
Therefore, if the initial condition is such that $T^0 < \hat{T}$ for $\hat{T}$ as in Eq.~\eqref{eq:T_hat}, then $T(t) \leq \hat{T}$ for all $t>0$. This finally proves the result.
\end{proof}

\begin{remark}
    The result on the positivity of the trajectories holds in general with a function $\theta$ given in terms of the transition rates given in Eq.~\eqref{eq:theta_pq} and \eqref{eq:properties_pq}.
\end{remark}

\begin{remark}
    The quantity $\tilde{R}$ in Eq.~\eqref{eq:Rtilde} can thus be considered as the actual carrying capacity for $R$.
\end{remark}


We now move to the investigation of the steady-states of the homogeneous model \eqref{eq:mainsyshom}. The following result holds.

\begin{prop}[Steady-states] \label{prop:equi}
For the parameter assumptions illustrated in Sec.~\ref{sec:model} and the function $\theta$ as in Eq.~\eqref{eq:pqt}, system \eqref{eq:mainsyshom} always admits the trivial equilibrium $E_0=(0,0)$ as well as a coexistence equilibrium $E_*=(R_*, T_*)$ whose components satisfy
\begin{equation} \label{eq:Tast}
    T_*=\dfrac{cdR_*}{k\hat{T}-csR_*}\hat{T},
\end{equation}
where $R_*$ is the positive solution to
\begin{equation} \label{eq:Rasteq}
(gs+\gamma d)R_*^2
 -\left( 2gs+\gamma d +gd  \right)\hat{R}R_*
+(g-d)(d+s)\hat{R}^2=0.
\end{equation}
\end{prop}
\begin{proof}
The steady-states associated with Eq.~\eqref{eq:mainsyshom} are the solutions to
\begin{subequations} \label{eq:mainsyshomeq}
\begin{align}
0 &= 
\left(g-\gamma \theta(T)\right)R\left(1-\dfrac{R}{\hat{R}}\right)
-\left(d+s\theta(T)\right)R,\label{eq:mshomeqR} \\[0.3cm] 
0 &= c\left(d+s\theta(T)\right)R-kT.  \label{eq:mshomeqT}
\end{align}
\end{subequations}
This system admits the trivial solution $E_0=(0,0)$ for all parameter values. Additional coexistence steady-states can be found by plugging into Eq.~\eqref{eq:mainsyshomeq} the expression for the function $\theta$ in \eqref{eq:theta_pq}. The case $T>T_*$ leads to no feasible equilibria, while we can find a coexistence equilibrium state when $T\geq T_*$ by first solving Eq.~\eqref{eq:mshomeqT} for $T$ (which leads to Eq.~\eqref{eq:Tast}) and then plugging this expression into Eq.~\eqref{eq:mshomeqR}. The resulting quadratic equation, shown in \eqref{eq:Rasteq}, admits two positive solutions which we correspondingly indicate with $R_*^-$ and $R_*^+$. In particular, it turns out that $T_*^+>\hat{T}$, and therefore not admissible in this case. On the other hand, we can prove that $0<R_*^-<\hat{R}$ and $0<T_*^-<\hat{T}$, i.e.~this steady-state lies in the rectangular invariant region defined in Prop.~\ref{prop:pos}. Therefore, the only biologically feasible coexistence equilibrium is $E_*^-=(R_*^-, T_*^-)$. In particular, thanks to the assumption in Eq.~\eqref{eq:T_hat}, we also have $T_*^-<R_*^-$. We can then define the coexistence state $E_*=E_*^-$, namely $E_*=(R_*, T_*)$ where
$$R_*=\dfrac{( 2gs+\gamma d +gd )-\sqrt{( 2gs+\gamma d +gd)^2-4(gs+\gamma d)(g-d)(d+s)}}{2(gs+\gamma d)}\hat{R},$$
and $T_*$ is given by Eq.~\eqref{eq:Tast}. It can also readily be shown that $R_*\leq \hat{R}$, and that the discriminant of equation~\eqref{eq:Rasteq} is nonnegative for all parameter values fulfilling the assumptions in Section~\ref{sec:model} (namely $g>d$, $g>\gamma$, $\hat{T}=c(d+s)\hat{R}/k$, and $c(d+s)/k<1)$, therefore $R_*\geq 0$. The first condition also implies that $0\leq T_*\leq \hat{T}$ for all parameter values.
\end{proof}



\begin{remark}
It can be shown that $R_*$ is monotonically decreasing with respect to both $s$ and $\gamma$. This is consistent with biological expectations: the higher the root sensitivity to autotoxicity, respectively the stronger the inhibition induced by autotoxicity, the less the biomass can grow.
\end{remark}

To verify the possible emergence of Turing patterns, we now investigate the linear stability analysis to homogeneous perturbations of the steady-states obtained in Prop.~\ref{prop:equi}.

\begin{prop} \label{prop:stabhom}
    The equilibrium $E_0$ of Eq.~\eqref{eq:mainsyshom} is unstable, whereas $E_*$ is asymptotically stable for the parameter values supporting its existence.
\end{prop}
\begin{proof}
The Jacobian matrix associated to system \eqref{eq:mainsyshom} reads
\begin{equation}
J(R,T)=
\begin{pmatrix} 
J_{11} & J_{12} \\ J_{21} & J_{22} \end{pmatrix},
\end{equation}
where
\begin{align*}
    J_{11}&=\left(g-\gamma\theta(T)\right)\left(1-\dfrac{2R}{\hat{R}}\right)-\left(d+s\theta(T)\right),\\
    J_{12}&=-\gamma R\, \theta'(T) \left(1-\dfrac{R}{\hat{R}}\right)-sR\, \theta'(T),\\
    J_{21}&=c\left(d+s\theta(T)\right),\\
    J_{22}&=cs R\, \theta'(T)-k.
\end{align*}
In $E_0$, the eigenvalues of the Jacobian matrix $J(0,0)$ are given by $\lambda_1=g-d$, $\lambda_2=-k$. Therefore, due to the assumption $g>d$, we have $\lambda_1>0$, $\lambda_2 < 0$, i.e.~this equilibrium is linearly unstable. At the coexistence steady-state $E_*$, on the other hand, we have $J_{11}^*,\,J_{12}^*,\,J_{22}^*<0$, whereas $J_{21}^*>0$. This implies that $J(R_*,T_*)$ admits a negative trace and a positive determinant for all parameter values satisfying the assumptions in Sec.~\ref{sec:model}; consequently, it is asymptotically stable.
\end{proof}


\subsection{Heterogeneous model}
We now turn our attention to the full model \eqref{eq:mainsys}. In particular, we investigate the stability of the coexistence steady-state $E_*$ to spatially heterogeneous perturbations. As known, in fact, Turing patterns emerge when an equilibrium that is linearly stable to homogeneous perturbations becomes unstable in the presence of spatially heterogeneous disturbances. Being $E_0$ unstable with respect to homogeneous perturbations (as shown in Prop.~\ref{prop:stabhom}), the trivial steady-state therefore cannot lead to Turing patterns. Moreover, other types of stable spatial patterns cannot emerge from $E_0$ since oscillations around bare soil would lead to unfeasible scenarios. Consequently, the trivial steady-state is excluded from our following analysis.\\
The following result holds.
\begin{prop} \label{prop:stabcross}
    The necessary condition for the instability of the coexistence steady-state ${E_*=(R_*,T_*)}$ with respect to heterogeneous perturbations is
    \begin{equation} \label{eq:condTur}
    \sigma > \sigma_L:=
    \left(\dfrac{k\hat{T}-csR_*}{cdR_*+kT_*}\right)d_R 
    + \dfrac{R_*}{\hat{R}}\left(\dfrac{d\hat{T}+sT_*}{(1-R_*/\hat{R})(cdR_*+kT_*)} \right)d_T.
\end{equation}
\end{prop}
\begin{proof}
The characteristic matrix $M_\kappa$ related to the $\kappa$-th eigenvalue of the Laplace operator with homogeneous Neumann boundary conditions, associated to system \eqref{eq:mainsys}, and evaluated at $E_*$ can be written as
\begin{equation}
    M_\kappa = J(R_*,T_*) -\lambda_\kappa J_\Delta(R_*,T_*) = \begin{pmatrix} 
J_{11}^*-\lambda_\kappa J^{\Delta*}_{11} & J_{12}^*-\lambda_\kappa J^{\Delta*}_{12} \\[0.3cm] 
J_{21}^* & J_{22}^*-\lambda_\kappa d_T
\end{pmatrix},
\end{equation}
where $J(R_*,T_*)$ is the Jacobian matrix for the homogeneous system and $J_\Delta(R_*,T_*)$ is the linearised diffusion matrix (both evaluated at $E_*$). In particular, the latter is defined as
\begin{equation*}
J_\Delta(R_*,T_*)=
\begin{pmatrix} 
J^{\Delta*}_{11} & J^{\Delta*}_{12} \\[0.2cm]
J^{\Delta*}_{21} & J^{\Delta*}_{22} 
\end{pmatrix}=
\begin{pmatrix} 
d_R-\sigma \theta(T_*) & -\sigma R_* \theta'(T_*)\\[0.3cm] 
0 & d_T
\end{pmatrix}.
\end{equation*}
For this matrix we have $J^{\Delta*}_{11},\,J^{\Delta*}_{22}>0$ and $J^{\Delta*}_{12}<0$.\\
In view of the assumption $\sigma<d_R$ and the consideration that $T_*<\hat{T}$, the trace of $M_\kappa$ remains negative, since
\[
\textnormal{tr}M_\kappa=\textnormal{tr}J(R_*,T_*)-\lambda_\kappa\left(d_T+d_R-\sigma \theta(T_*)\right)<0.
\]
On the other hand, the determinant of $M_\kappa$ reads
\[
\det M_\kappa=
d_T J^{\Delta*}_{11} \lambda_\kappa^2
-\left( d_T J^{*}_{11}+ J^{\Delta*}_{11}J^{*}_{22}
-J^{\Delta *}_{12}J^{*}_{21}\right) \lambda_\kappa
+\det J(R_*,T_*).
\]
Since $d_T>0$ and $\det (R_*,T_*)>0$ (due to Prop.~\ref{prop:stabhom}), in order for $E_*$ to become unstable -- thus leading to Turing instability -- the following necessary condition must hold
\[
d_T J^{*}_{11}+ J^{\Delta*}_{11}J^{*}_{22} -J^{\Delta *}_{12}J^{*}_{21}>0.
\]
Expressing all the Jacobian components explicitly in terms of $R_*$, $T_*$, we obtain Eq.~\eqref{eq:condTur}.
\end{proof}

\begin{remark}
    A particularly interesting scenario from the biological viewpoint (which potentially satidfies~\eqref{eq:condTur}) is given by $d_T \ll 1$ and $\sigma \sim d_R$; this in fact corresponds to a low propagation rate for the exposed roots and low diffusion for the toxicity.
\end{remark}

\begin{remark}
    In Sec.~\ref{sec:numcont}, we will numerically investigate the monotonicity of the critical threshold $\sigma_L$ depending on $s$ and $\gamma$ (and suitably fixing all other parameter values). In particular, we aim to show how growth-inhibition and extra-mortality affect the range $(\sigma_L, d_R)$ of $\sigma$-values supporting the emergence of Turing patterns.
\end{remark}

The most relevant consequence of Prop.~\ref{prop:stabcross} is that when $\sigma=0$ (namely, no propagation-reduction effects are considered) the condition in Eq.~\eqref{eq:condTur} does not hold and Turing patterns cannot emerge (the reaction part does not present the activator--inhibitor structure so that no Turing instability can appear with standard diffusion terms). In other words, cross-diffusion in system~\eqref{eq:mainsys} is crucial to retrieve Turing patterns in a model coupling only biomass and toxicity dynamics (without an explicit equation for the water variable).



We finally observe that when both $s$ and $\gamma$ vanish (i.e.~in the scenario where autotoxicity only acts on the propagation mechanism whereas both extra-mortality and growth-inhibition effects are considered negligible), condition \eqref{eq:condTur} reduces to
\begin{equation} \label{eq:condTursgam0}
\sigma >\sigma_L^0:=\frac{g}{2} \left(\frac{d_R}{g-d}+\frac{d_T}{k}\right).
\end{equation}




\section{Numerical investigation}\label{sec:sim}

In this section, we carry out a thorough numerical investigation of Eq.~\eqref{eq:mainsys}--\eqref{eq:BICfr}, focusing our attention on the influence of the autotoxicity effects (growth-inhibition, extra-mortality, and propagation-reduction) on the biomass dynamics. Our strategy is twofold: on the one hand, we perform direct simulations of the model for both one- and two-dimensional domains for different values of $\gamma$, $s$, and $\sigma$ fulfilling the parameters assumptions. On the other hand, we adopt the continuation software \texttt{pde2path}~\cite{pde2path} to compute the bifurcation diagrams in the 1D case with respect to the bifurcation parameters~$s$ and $\sigma$, revealing the complex structure and variety of patterns supported by our cross-diffusion system.\\
As our main aim is to study the toxicity parameters $\gamma$, $s$, and $\sigma$, we fix the other parameter values (unless stated otherwise) as follows
\begin{equation} \label{eq:parval}
\begin{gathered}
 g = 10, \quad c = 0.5, \quad \hat{R} = 6, \quad d = 1, \quad k = 1, \quad \hat{T} = 3(1+s), \quad d_R = 3.33, \quad d_T= 0.05,
\end{gathered}
\end{equation}
where the units of measure can be found in Tab.~\ref{tab:pars}.
In particular, we note that the precise value of~$\hat{T}$ depends on $s$, in order to fulfill the parameter assumptions outlined in Sec.~\ref{sec:model}--\ref{sec:TI}. \\
Even though our theoretical investigation considers a general rather than a specific plant species, the units and parameter values used in this work are in line with available data in the literature (see e.g.~\cite{De_Baets_2007,martinez1998belowground}). We also remark that the values of $g$, $c$, $\hat{R}$, $d$, and $k$ are in agreement with \cite{Carteni_2012}. 

\subsection{Numerical simulations}

In this section, we perform direct simulations of Eq.~\eqref{eq:mainsys}--\eqref{eq:BICfr} for $\Omega \subset \mathbb{R}^n$ with $n=1,2$. In particular, we choose
\begin{equation}
\begin{aligned}
    \Omega &= [0,L] \qquad &&\text{ for } n=1, \\
    \Omega &= [0,L] \times [0,L]  \qquad &&\text{ for } n=2,
\end{aligned}
\end{equation}
with length $L=8$ metres for $t \in [0, t_{\rm fin}]$ with $t_{\rm fin}=1000$ years. 

In particular, we compare the one- and two-dimensional profiles obtained at $t_{\rm fin}$ for the biomass $R$ and the toxicity $T$ for four different scenarios: three, where at least one of the three effects induced by toxicity -- namely, growth-inhibition, extra-mortality, and propagation-reduction -- is neglected, and one where all effects act simultaneously. As these effects are analytically represented by the parameters $\gamma$, $s$, and $\sigma$ respectively, we focus on the following parameter sets (the other parameter values are fixed as in Eq.~\eqref{eq:parval}):
\begin{enumerate}[(i)]
    \item\label{scenario:i} no propagation-reduction effect: $\gamma=0.1$, $s=0.5$, $\sigma=0$,
    \item\label{scenario:ii} all effects: $\gamma=0.1$, $s=0.5$, $\sigma=3$,
    \item\label{scenario:iii} no growth-inhibition effect: $\gamma=0$, $s=0.5$, $\sigma=3$, 
    \item\label{scenario:iv} no extra-mortality effect: $\gamma=0.1$, $s=0$, $\sigma=3$.
\end{enumerate}
In our simulations, we use a numerical scheme based on finite differences in space and forward Euler in time. The spatial domain $\Omega$ is discretised via:
\begin{itemize}
    \item[-] a grid of $400$ elements with spacing $\delta x = 0.02$ in the one-dimensional case,
    \item[-] a lattice of $80 \times 80$ elements with spacing $\delta x = 0.1$ in the two-dimensional case.
\end{itemize}
In both cases, the timestep is chosen as $\delta t=10^{-5}$. \\
For the one-dimensional simulations, the initial condition satisfies
\begin{equation}
    R^0(x) = 10 \, e^{-\frac{25 (L-2 x)^2}{4 L}}, \quad T^0(x) = 0.
\end{equation}
As for the two-dimensional case, the initial condition consists in four squares of four pixels per side randomly placed in the lattice with uniform value of $R$ randomly chosen in the range $1-2 \, \textrm{kg} \, \textrm{m}^{-2}$ and $T$ uniformly zero.




Figure~\ref{fig:comp_RT_1D} shows the steady states of the roots biomass (left) and toxicity (right) obtained in the simulations of Eq.~\eqref{eq:mainsys}--\eqref{eq:BICfr} over the one-dimensional domain $\Omega = [0,L]$ for the four scenarios (i)--(iv) outlined above. Here, scenario~\eqref{scenario:i} is represented by a dotted-dashed line, scenario~\eqref{scenario:ii} by a solid line, scenario~\eqref{scenario:iii} by a dashed line, and scenario~\eqref{scenario:iv} by a dotted line. We note that for this specific set of parameter values, the patterns associated with scenario~\eqref{scenario:ii} and scenario~\eqref{scenario:iv} are identical, leading to overlapping curves. Moreover, as predicted in Prop.~\ref{prop:stabcross}, the parameter $\sigma$ related to the cross-diffusion term in \eqref{eq:mainsys} plays a crucial role in the formation of spatial patterns: only for $\sigma=0$ (scenario~\eqref{scenario:ii}), the simulation returns a homogeneous state, whereas in all other cases (where $\sigma \neq 0$) a nontrivial pattern arises -- even in scenario~\eqref{scenario:i}, where both $\gamma$ and $s$ are turned off. 
The biomass and autotoxicity patterns are in phase with respect to the pattern structure; however, the toxicity peaks are located in correspondence with the centre of the biomass gaps. This is due to the combination of the healthy and exposed root dynamics, whose local maxima are in- and out-of-phase, respectively.
Finally, the biomass--water interactions described by our cross-diffusion model lead to the formation of double-peak pulses or 2-front vegetation spots (also observed in a reduced version of the biomass--water model by Gilad et al., see \cite{Jaibi_2020}). 

\begin{figure}[H]
    \centering
    \begin{subfigure}[t]{0.5\textwidth}
    \centering
    \begin{overpic}[scale=0.5]{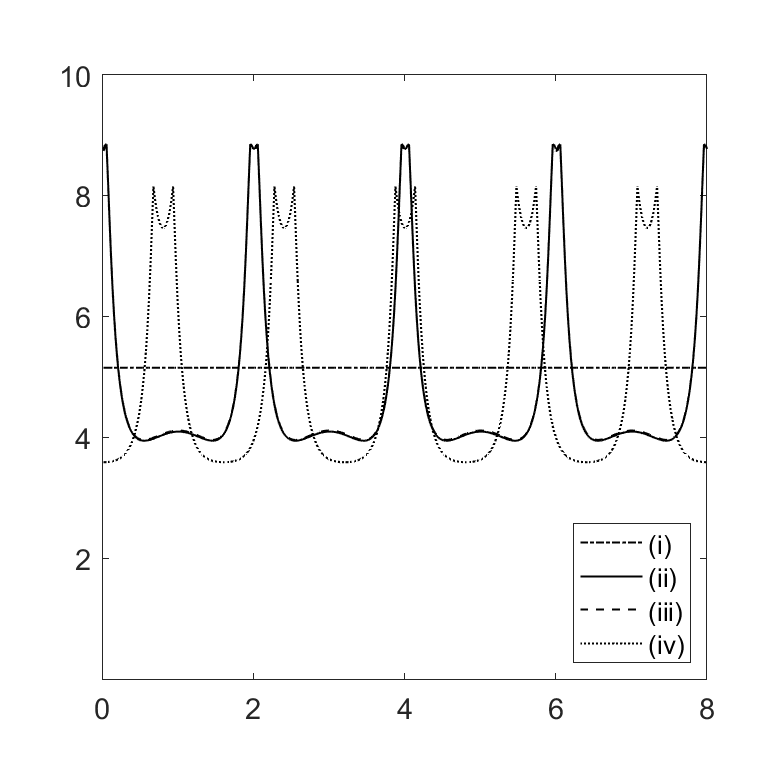}
    \put(2,50){$R$}
    \put(50,2){$x$}
    \end{overpic} 
    \caption{}
    \end{subfigure}
    \hspace{-0.3cm}
    \begin{subfigure}[t]{0.5\textwidth}
    \centering
    \begin{overpic}[scale=0.5]{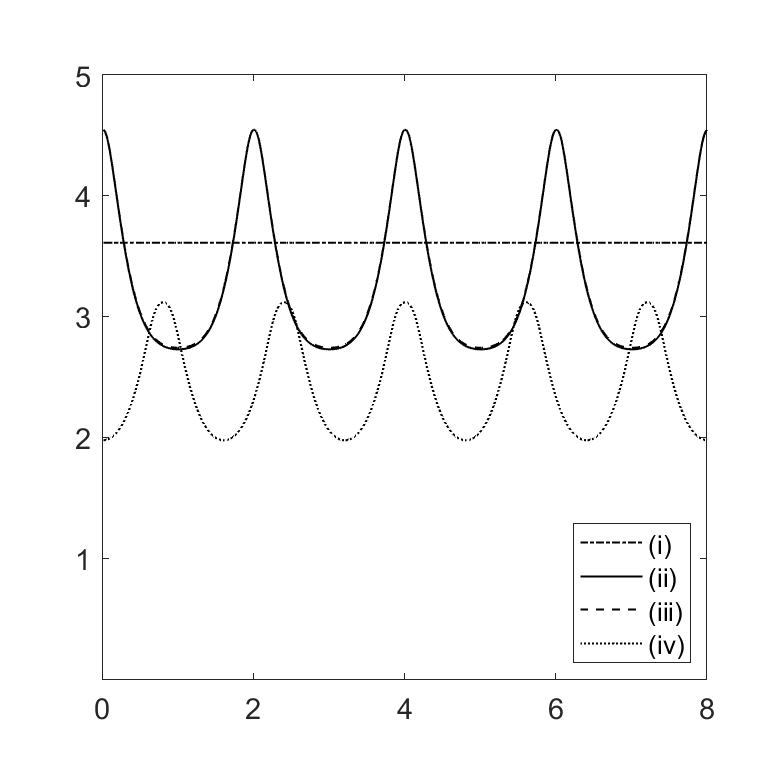}
    \put(2,50){$T$}
    \put(50,2){$x$}
    \end{overpic} 
    \caption{}
    \end{subfigure}
    \caption{Result of simulations of Eq.~\eqref{eq:mainsys}--\eqref{eq:BICfr} on a one-dimensional domain $\Omega$ of length $L=8$ metres at $t=t_{\rm fin}$ for (a) the roots biomass $R$ and (b) the toxicity concentration $T$ for the four different scenarios~\eqref{scenario:i},~\eqref{scenario:ii},~\eqref{scenario:iii},~\eqref{scenario:iv} corresponding to dotted-dashed, solid, dashed, and dotted lines, respectively. Other parameter values are fixed as in Eq.~\eqref{eq:parval}. 
    }
    \label{fig:comp_RT_1D}
\end{figure}



Simulations in the two-dimensional case (closer to real-life scenarios) confirm the convergence of Eq.~\eqref{eq:mainsys}--\eqref{eq:BICfr} to a stable state. When cross-diffusion is present (i.e. for $\sigma \neq 0$), both $R$ and $T$ reach a patterned configuration (see Figure~\ref{fig:pannello2D}(ii)--(iv)), whereas for $\sigma=0$ the configuration is spatially homogeneous (see Figure~\ref{fig:pannello2D}(i)). In particular, in the simulations where patterns emerge, we observe that biomass spots are irregular and equally spaced in a uniform non-zero biomass background, and exhibit a more pronounced border with a small depression in the centre (corresponding to the double-peak pulses in the 1D simulations). Correspondingly, the peaks of the toxicity spots are located at the centre of the biomass spots. \\
The formation of these patterns occurs on a long timescale; in the Supplementary Material, we provide videos which reveal in more detail how the spots rearrange to be equally spaced. The main features of the patterns do not change in scenarios~\eqref{scenario:ii} and~\eqref{scenario:iii}, whereas in scenario~\eqref{scenario:iv} we observe larger, more abundant spots.

\begin{figure}[!ht]
\centering
\vspace{.3cm}
\begin{overpic}[scale=0.7]{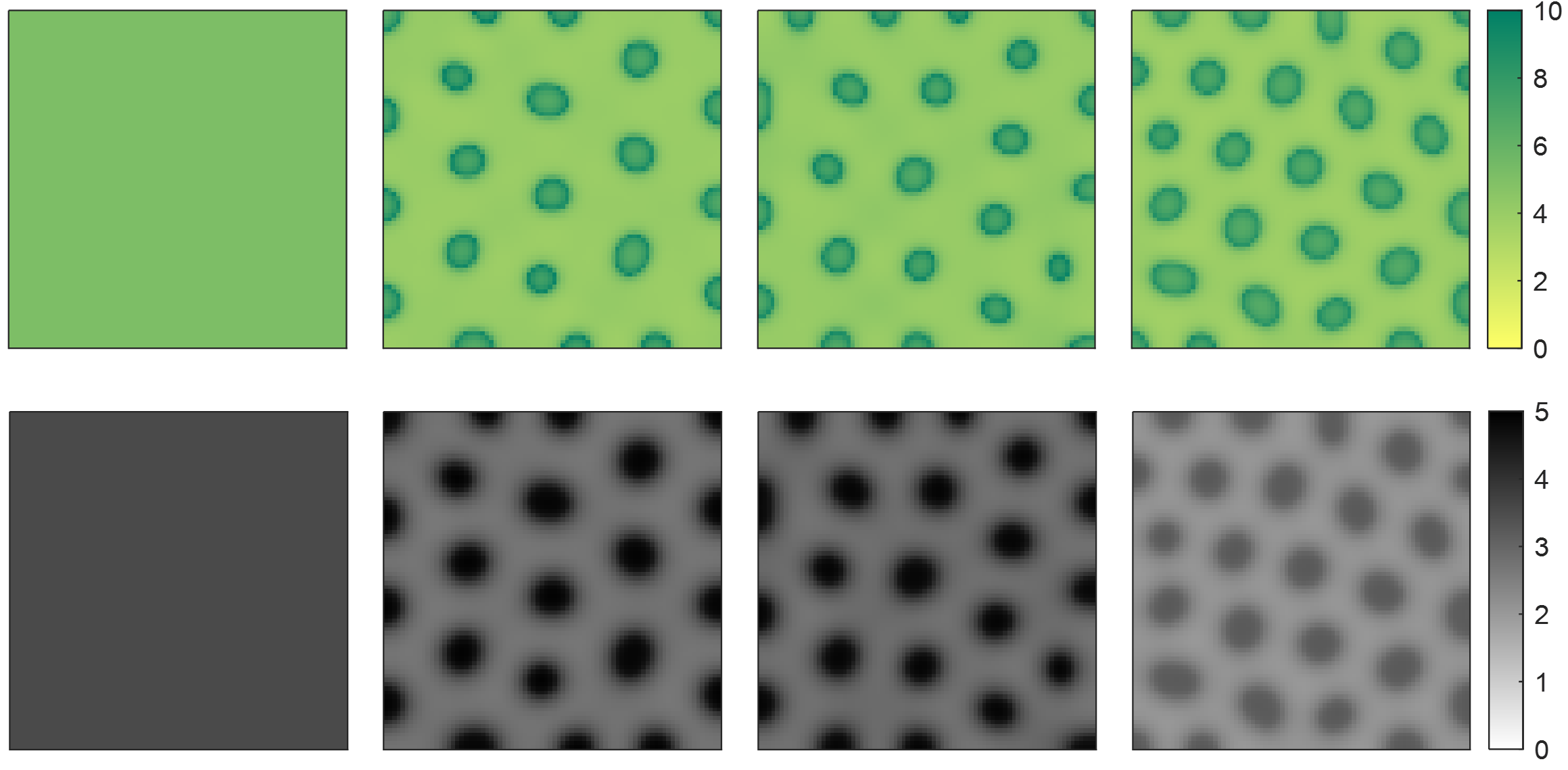}
  	\put(10,50){(i)}
  	\put(33,50){(ii)}
  	\put(56,50){(iii)}
        \put(81,50){(iv)}
        \put(-4,10){$T$}
        \put(-4,37){$R$}
    \end{overpic}
\vspace{.3cm}
    \caption{Result of simulation of Eq.~\eqref{eq:mainsys}--\eqref{eq:BICfr} on a two-dimensional domain squared domain $\Omega$ ($L=8$ metres) at $t=t_{\rm fin}$ for the roots biomass $R$ (upper panel) and toxicity concentration $T$ (lower panel) for the four different scenarios~\eqref{scenario:i},~\eqref{scenario:ii},~\eqref{scenario:iii},~\eqref{scenario:iv} corresponding to the first, second, third, and fourth column, respectively. Other parameter values are fixed as in Eq.~\eqref{eq:parval}.}
    \label{fig:pannello2D}
\end{figure}
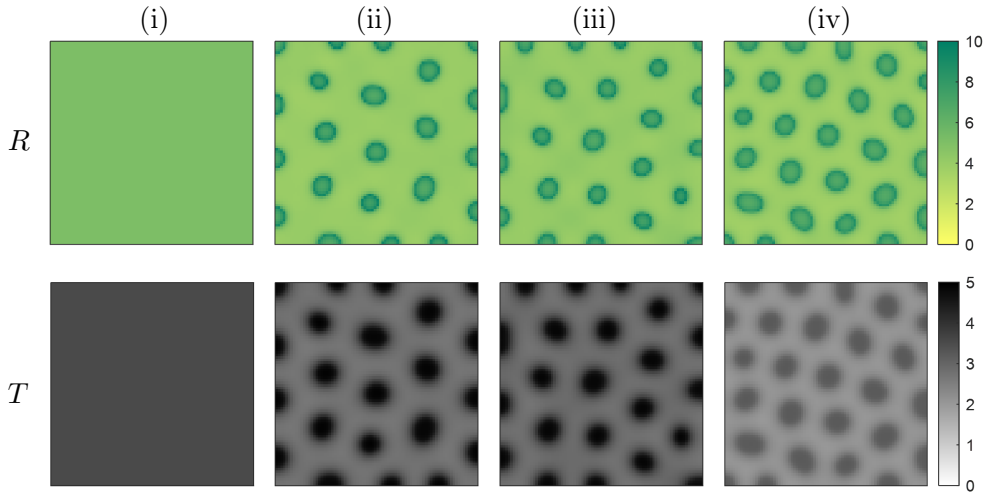

\subsection{Bifurcations and numerical continuation} \label{sec:numcont}

In this section, we numerically investigate the effect of crucial parameters on the possibility of having Turing patterns and on the bifurcation structure of steady states. From conditions~\eqref{eq:condTur}, we know that growth-inhibition (represented by the parameter $\gamma$) and extra-mortality (represented by $s$) changes the critical value $\sigma_L$ in \eqref{eq:condTur} describing the infimum of the interval of $\sigma$-values supporting the occurrence of Turing patterns. In particular, we fix all parameter values except for $\gamma$ and $s$ as in Eq.~\eqref{eq:parval} and plot the threshold $\sigma_L$ as a function of these two parameters in Figure~\ref{fig:rhssigma}. We observe that the area of the regions obtained by considering a cross-section $(\bar{\gamma},s,\sigma)$ in $(\gamma,s,\sigma)$-space for a fixed $\bar{\gamma}\in[0,10]$ is monotonically increasing with respect to $\gamma$; this also applied to the area of the regions obtained by considering a cross-section $(\gamma,\bar{s},\sigma)$ in $(\gamma,s,\sigma)$-space for a fixed $\bar{s}\in[0,1]$, except for the (small) range $s \in [0,s^\ast]$ (with $s^\ast\approx 0.16$) where the area is monotonically increasing.
This result suggests that the growth-inhibition and extra-mortality effect play against the propagation-reduction mechanism induced by cross-diffusion in the formation of Turing patterns by reducing the size of the region supporting them for a wide range of parameter values.

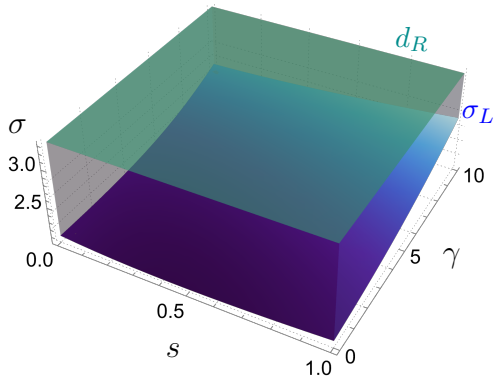
\begin{figure}[!ht]
\centering
\begin{overpic}[scale=0.5]{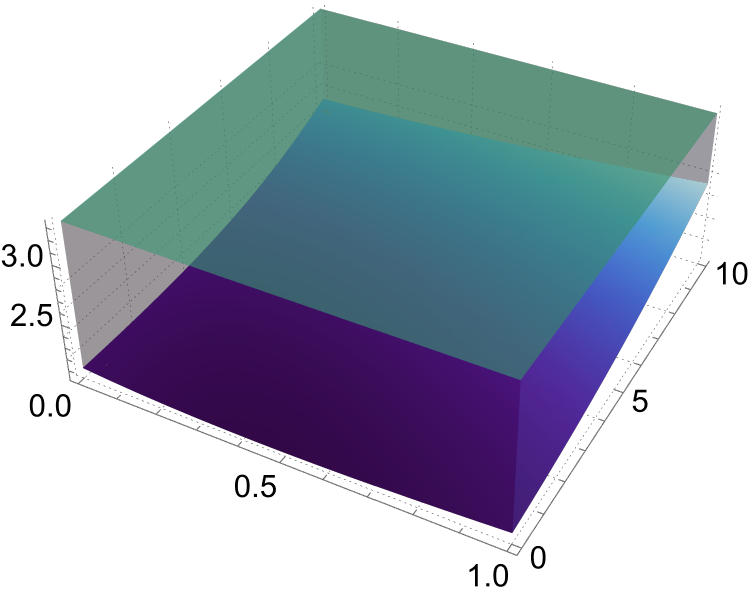}
  	\put(33,5){\large{$s$}}
        \put(81,70){\textcolor{darkcyan}{$d_R$}}
        \put(95,55){\textcolor{blue}{$\sigma_L$}}
        \put(91,25){\large{$\gamma$}}
        \put(0,52){\large{$\sigma$}}
\end{overpic}
\caption{Region of admissible $\sigma$-values as a function of $\gamma$ and $s$ (describing the growth-inhibition and extra-mortality effect, respectively) in $(\gamma,s,\sigma)$-space (all other parameter values are fixed as in Eq.~\eqref{eq:parval}). This grey region is delimited by the two-dimensional surfaces $\sigma=\sigma_L$ (as defined in the Turing condition \eqref{eq:condTur}) and $\sigma=d_R$, in blue and cyan respectively.}
    \label{fig:rhssigma}
\end{figure}

The analytical study and simulations can be complemented by numerical continuation results that show the bifurcation structure of steady states. To this end, we exploit the numerical continuation software \texttt{pde2path} to compute the bifurcation diagrams of stationary solutions with respect to different parameters. How to implement the cross-diffusion system into \texttt{pde2path} is explained in~\cite{CKCS}. In the following bifurcation diagrams, we adopt the standard conventions for the representation of bifurcation branches, namely thin/thick lines represent unstable/stable branches respectively, while dots denote bifurcation points.

Figure~\ref{fig:bif_diag} shows the one-parameter bifurcation diagrams of steady-state solutions in the 1D case with respect to parameters $\sigma$ and $s$ respectively. In particular, Figure~\ref{fig:bif_sigma} refers to scenario~\eqref{scenario:ii}, while Figure~\ref{fig:bif_s} to scenario~\eqref{scenario:iii}. In both cases, as reference quantity on the $y$-axis, we use the $L^1$-norm of $R$ being the total amount of biomass on the domain. We also highlight with a dot the point on the branches corresponding to the steady-state solutions in Figure~\ref{fig:comp_RT_1D}.

In Figure~\ref{fig:bif_sigma}, we can see that the homogeneous state $E_*$ (black line) becomes unstable for sufficiently large values of $\sigma$ and branches of nonhomogeneous solutions arise at the bifurcation points. The second branch (orange) is bifurcating subcritically and then becomes stable at the fold. This branch is the one corresponding to the stable solutions (solid line) in Figure~\ref{fig:comp_RT_1D}. From Figure~\ref{fig:bif_s}, we observe that sufficiently large values for $s$ stabilise the homogeneous state $E_*$ (black line). As in the previous case, the orange branch corresponds to the stable solutions (dashed line) in Figure~\ref{fig:comp_RT_1D}. 

In both cases, since the orange branch is bifurcating subcritically, the system exhibits multistability regions in which both the homogenous solution and the nonhomogenous one are stable. In the Supplementary Material, we provide videos that show how the solution on the orange branches modifies when the bifurcation parameter is varied, and also that the region of double-peaked patterns is not narrow.


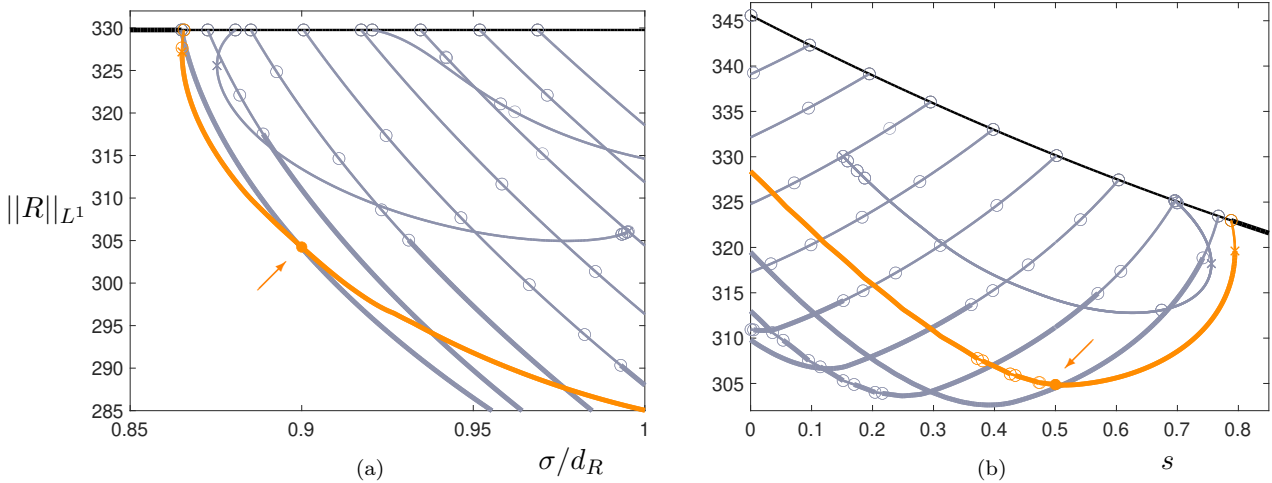
\begin{figure}[!h]
    \centering
    \begin{subfigure}[t]{0.5\textwidth}
    \centering
    \begin{overpic}[scale=0.45]{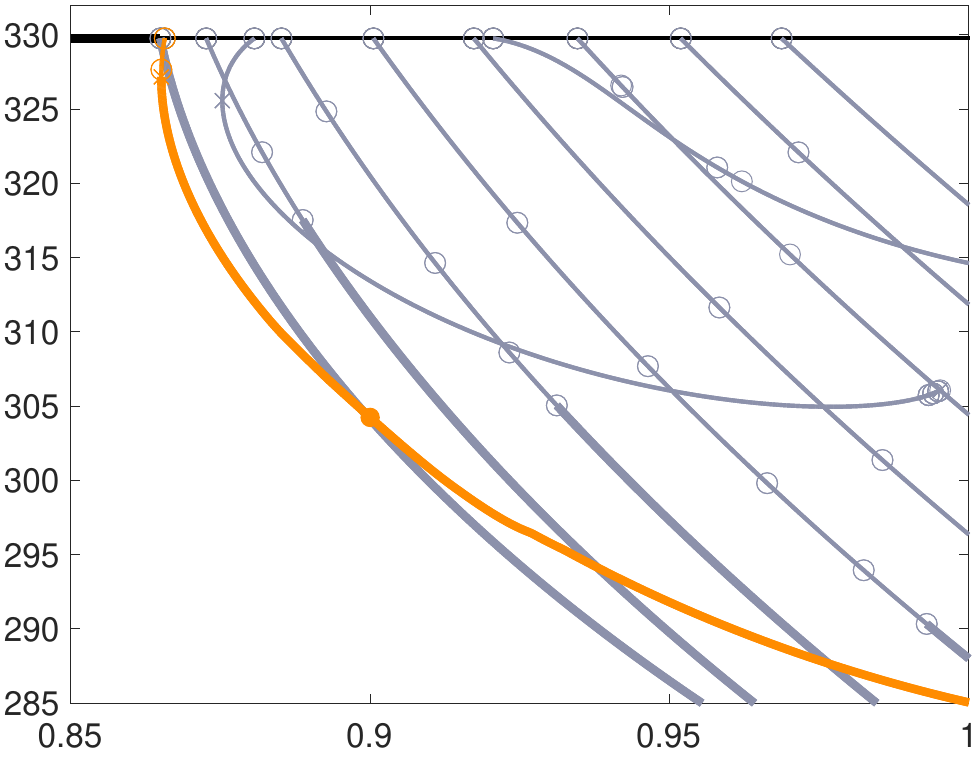}
  	\put(80,-5){$\sigma/d_R$}
  	\put(-15,40){$||R||_{L^1}$}
        \put(30,27){\color{orange}\vector(1,1){5}}
    \end{overpic}
    \caption{\label{fig:bif_sigma}}
    \end{subfigure}
    \hspace{-0.5cm}
    \begin{subfigure}[t]{0.5\textwidth}
    \centering
    \begin{overpic}[scale=0.45]{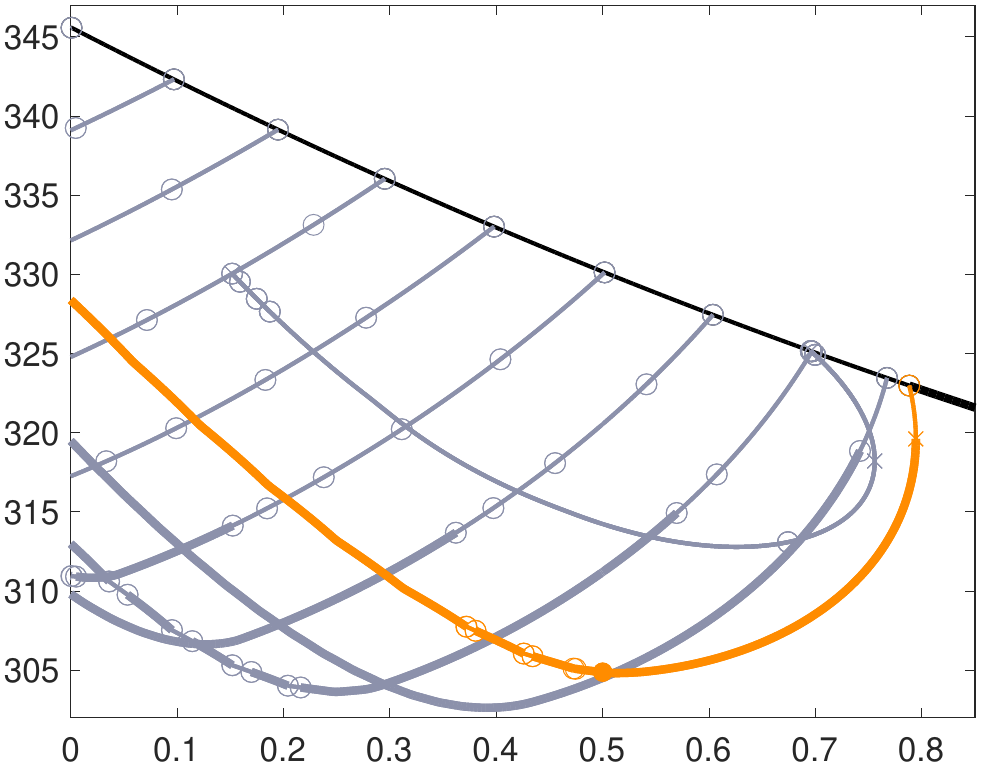}
  	\put(80,-5){$s$}
        \put(68,18){\color{orange}\vector(-1,-1){5}}
    \end{overpic}
    \caption{\label{fig:bif_s}}
    \end{subfigure}
    \caption{Bifurcation diagram with respect to (a) $\sigma$ in scenario~\eqref{scenario:ii} and (b) parameter $s$ in  scenario~\eqref{scenario:iii}. The black line corresponds to the homogenous branch, while branches of nonhomogenous solutions are shown in grey and orange. In (a), the orange branch is the second branch bifurcating subcritically from the homogeneous branch. In (b), the orange branch is the first branch bifurcating subcritically from the homogeneous branch. Solutions in scenario~\eqref{scenario:ii} and in scenario~\eqref{scenario:iii} in Figure~\ref{fig:comp_RT_1D} 
    corresponds to the markers on the orange branches in (a) at $\sigma/d_R=0.9$ and in (b) at $s=0.5$.}
    \label{fig:bif_diag}
\end{figure}



\section{Concluding remarks}\label{sec:conclusion}

In this work, we have developed and analysed a new cross-diffusion model for vegetation dynamics that focuses only on biomass--water interactions. Differently from previous biomass--water models, the three effects induced by autotoxicity on the biomass dynamics -- namely growth-inhibition, extra-mortality, and propagation-reduction -- have been derived from fast-reaction limits. The novelty of the paper has an impact on different aspects. On the one hand, the cross-diffusion model has been justified by the exploitation of time-scale arguments and led to the particular form of the cross-diffusion term: instead of what happens in the SKT model, where the presence of the other species causes extra movements of individuals, here autotoxicity leads to a reduction in the corresponding cross-diffusion term of the biomass equation. On the other hand, an additional crucial novelty of our model consists in the presence of stable Turing patterns even without taking into account water dynamics; growth-inhibition and extra-mortality effects, in fact, are alone insufficient to support the emergence of such patterns, as shown in several examples in the literature (see e.g.~\cite{Carteni_2012, Salvatori_2023}). Retrieving Turing patterns by sole means of a cross-diffusion mechanism between biomass and autotoxicity -- in the absence of explicit plant-water feedback -- thus broadens the investigation of vegetation spatial structures beyond arid ecosystems. We remark that the parameter values leading to spatial patterns were here fixed in agreement with other studies~\cite{Carteni_2012, De_Baets_2007, martinez1998belowground}, rather than chosen specifically for our model. 
Among the stable patterns observed in our numerical investigation (both by means of direct simulation and continuation), it is worth noting the presence of double-peak pulses of biomass in one-dimensional domains -- which are equivalent to reinforced boundaries for the irregular spots forming on two-dimensional domains -- for a wide region of parameter space. This type of patterns (also found e.g.~in \cite{Bera_2025,Jaibi_2020, rottschafer1998transition}) are observed in various ecosystems, see e.g.~\cite{Salvatori_2023,Yang_2018,Zotti_2025}. Moreover, we remark that, differently from standard Turing patterns induced by biomass--water feedback, the biomass density between the spots is not zero, mimicking the presence of uniform vegetation rather than bare soil within the pattern.
As observed in our analysis and confirmed by the thorough investigation of the system's bifurcations performed in \texttt{pde2path}, the formation of these stable, complex patterns occurs for a large set of parameter values, i.e.~is an inherent feature of our cross-diffusion model. We also remark that the results shown here for $\theta(T)$ as the piecewise function in Eq.~\eqref{eq:pqt} are valid also for a smooth function of Holling III-type as long as $\theta(T_*) \approx T_*/\hat{T}$ and $\theta'(T_*) \approx 1/\hat{T}$. \\
In the future, we plan to undertake several research directions. First, we aim to perform a weakly nonlinear analysis, in order to derive the equation governing the amplitude of the patterns and predict the type of emerging bifurcations (as in \cite{solopaper} for the SKT model). By looking at the neutral stability curves, we can also identify a possible location in the parameter space of Hopf bifurcations, possibly leading to time-periodic patterns. Moreover, as we expect the model to also exhibit travelling patterns because of its intrinsic multiscale structure, we aim to apply Geometric Singular Perturbation Theory (GSPT) to our cross-diffusion system; this might require an extension of the available techniques in the field, as they have not been applied to this type of systems yet (to the best of our knowledge). 
Finally, a particularly interesting research goal from the ecological and modelling viewpoint consists in considering an additional plant species to investigate how competition for resources and plant--soil feedback can influence (and potentially foster) spatial patterns of biodiversity.

\bigskip

\section*{Data availability} 
The MATLAB code used to perform both direct simulations and numerical continuation is available on GitHub at \href{https://github.com/aiuorio28/crosstox}{https://github.com/aiuorio28/crosstox}.


\section*{Acknowledgements} 
FG is a member of the INdAM-GNCS, while AI and CS are members of the INdAM-GNFM. AI and CS have received partial support from CIRM in the framework of the ``Research in pairs'' programme. CS has received funding from the Programma Giovani Ricercatori ``Rita Levi Montalcini'' 2021. CS was partially supported by the University of Parma through
the action Bando di Ateneo 2022 per la ricerca co-funded by MUR-Italian Ministry of Universities and Research - D.M. 737/2021 - PNR - PNRR - NextGenerationEU (project:
``Collective and self-organised dynamics: kinetic and network approaches''). FG was partially supported by PRIN 2022 PNRR P2022WC2ZZ ``A multidisciplinary approach to evaluate ecosystems resilience under climate change''. Finally, the authors thank Fabrizio Cartenì for the fruitful discussions that significantly contributed to the development of this work.

\bibliography{references}{}
\bibliographystyle{plain} 

\clearpage

\end{document}